\documentclass[12pt]{article}

\usepackage{enumerate}
\usepackage{amsmath}
\usepackage{amssymb,latexsym}
\usepackage{amsthm}
\usepackage{color}
\usepackage{graphicx}
\usepackage{hyperref}
\usepackage{amscd}
\usepackage{graphicx}
\usepackage{array}

\DeclareMathOperator{\Tr}{Tr}
\DeclareMathOperator{\N}{N}

\title{On maximum additive Hermitian rank-metric codes}
\author{Rocco Trombetti and Ferdinando Zullo}
\date{}

\newcommand{\cC}{{\mathcal C}}

\newcommand{\C}{{\mathcal C}}
\newcommand{\cH}{{\mathcal H}}

\newcommand{\F}{{\mathbb F}}

\newcommand{\Aut}{\mathrm{Aut}}

\newcommand{\la}{\langle}
\newcommand{\ra}{\rangle}

\newcommand{\rk}{\mathrm{rk}}

\newtheorem{theorem}{Theorem}[section]

\newtheorem{lemma}[theorem]{Lemma}
\newtheorem{corollary}[theorem]{Corollary}

\newtheorem{remark}[theorem]{Remark}

\DeclareMathOperator{\PG}{\mathrm{PG}}

\newcommand\qbin[3]{\left[\begin{matrix} #1 \\ #2 \end{matrix} \right]_{#3}}

\begin{document}

\maketitle

\begin{abstract}
Inspired by the work of Zhou ``On equivalence of maximum additive symmetric rank-distance codes'' (2020) based on the paper of Schmidt ``Symmetric bilinear forms over finite fields with applications to coding theory'' (2015), we investigate the equivalence issue of maximum $d$-codes of Hermitian matrices.
More precisely, in the space $\mathrm{H}_n(q^2)$ of Hermitian matrices over $\F_{q^2}$ we have two possible equivalence: the classical one coming from the maps that preserve the rank in $\F_{q^2}^{n\times n}$, and the one that comes from restricting to those maps preserving both the rank and the space $\mathrm{H}_n(q^2)$.
We prove that when $d<n$ and the codes considered are maximum additive $d$-codes and $(n-d)$-designs, these two equivalence relations coincide. As a consequence, we get that the idealisers of such codes are not distinguishers, unlike what usually happens for rank metric codes.
Finally, we deal with the combinatorial properties of known maximum Hermitian codes and, by means of this investigation, we present a new family of maximum Hermitian $2$-code, extending the construction presented by Longobardi et al. in ``Automorphism groups and new constructions of maximum additive rank metric codes with restrictions'' (2020).
\end{abstract}

\\

\\

{\footnotesize\thanks{This research was partially supported by the Italian National Group for Algebraic and Geometric Structures and their Applications (GNSAGA - INdAM). The last author was also supported by the project ''VALERE: VAnviteLli pEr la RicErca" of the University of Campania ''Luigi Vanvitelli''.}}

\section{Introduction}

Let us consider $\F_q^{n\times n},$ the set of the square matrices of order $n$ defined over $\F_q$, with $q$ a prime power.
It is well-known that $\F_q^{n\times n}$ equipped with
\[ d(A,B)=\rk (A-B), \]
where $A,B \in \F_q^{n\times n}$, is a metric space.
If $\mathrm{C}$ is a subset of $\F_q^{n\times n}$ with the property that for each $A,B \in \mathrm{C}$ then $d(A,B)\geq d$ with $1\leq d\leq n$, then we say that $\mathrm{C}$ is a $d$-\emph{code}.
Furthermore, we say that $\mathrm{C}$ is \emph{additive} if $\mathrm{C}$ is an additive subgroup of $(\F_q^{n\times n},+)$, and $\mathrm{C}$ is $\F_q$-\emph{linear} if $\mathrm{C}$ is an $\F_q$-subspace of $(\F_q^{n\times n},+,\cdot)$, where $+$ is the classical matrix addition and $\cdot$ is the scalar multiplication by an element of $\F_q$.
Delsarte in \cite{Delsarte} shows the following bound for a $d$-code $\C$
\[ |\mathrm{C}| \leq q^{n(n-d+1)}, \]
known as \emph{Singleton like bound}, see also \cite{Gabidulin}.
Codes whose parameters satisfy the aforementioned bound are known as \emph{maximum rank distance codes} (or shortly \emph{MRD}-codes), and they have several important applications.
Attention has been paid also to rank metric codes with restrictions, which are codes whose words are alternating matrices \cite{DelsarteGoethals}, symmetric matrices \cite{LLTZ,Schmidt2010,Schmidt2015,Zhou} and Hermitian matrices \cite{Schmidt2018}.

In this paper we deal with Hermitian matrices over $\F_{q^2}$.

Consider $\overline{\cdot} \colon x \in \F_{q^2}\mapsto x^q\in \F_{q^2}$ the conjugation map over $\F_{q^2}$.
Let $A\in \F_{q^2}^{n\times n}$ and denote by $A^*$ the matrix obtained from $A$ by conjugation of each entry and transposition.
A matrix $A\in \F_{q^2}^{n\times n}$ is said \emph{Hermitian} if $A^*=A$.
Denote by ${\rm H}_{n}(q^2)$ the set of all Hermitian matrices of order $n$ over $\F_{q^2}$.
In \cite[Theorem 1]{Schmidt2018}, Schmidt proved that if $C$ is an additive $d$-code contained in ${\rm H}_{n}(q^2)$, then
\begin{equation}\label{eq:SingletonBoundHermitian}
|\mathrm{C}| \leq q^{n(n-d+1)}.
\end{equation}
When the parameters of $\mathrm{C}$ satisfy the equality in this bound, we say that $\mathrm{C}$ is a \emph{maximum} (additive) Hermitian $d$-code.
Schmidt also provided constructions of maximum $d$-codes for all possible value of $n$ and $d$, except if $n$ and $d$ are both even and $3<d<n$ \cite[Theorems 4 and 5]{Schmidt2018}. When $d=2$ and when $d=n$, it is easy to exhibit constructions of maximum additive $d$-codes. For instance, when $d=n$ a {\it semifield spread set} of symmetric $n\times n$ matrices over $\F_q$, gives rise to an example of maximum $n$-code of  ${\rm H}_n(q^2)$. For $d=2$, instead, we can take all matrices in ${\rm H}_n(q^2)$ whose main diagonal contains only zeros.

For given $a\in \F_q^*$, $\rho \in \mathrm{Aut}(\F_{q^2})$, $A \in \mathrm{GL}(n,q^2)$ and $B \in {\rm H}_{n}(q^2)$, the map
\begin{equation}\label{eq:theta}
\Theta\colon C \in {\rm H}_{n}(q^2) \mapsto aAC^\rho A^*+B \in {\rm H}_{n}(q^2),
\end{equation}
where $C^\rho$ is the matrix obtained from $C$ by applying $\rho$ to each of its entry, preserves the rank distance and conversely, see \cite{Wan}. For two subset $\mathrm{C}_1$ and $\mathrm{C}_2$ of ${\rm H}_{n}(q^2)$, if there exists $\Theta$ as in \eqref{eq:theta} such that
\[ \mathrm{C}_2=\{\Theta(C) \colon C \in \mathrm{C}_1\} \]
we say that $\mathrm{C}_1$ and $\mathrm{C}_2$ are \emph{equivalent in} ${\rm H}_{n}(q^2)$.
Nevertheless, we may consider the maps of $\F_{q^2}^{n\times n}$ preserving the rank distance, which by \cite{Wan} are all of the following kind
\begin{equation}\label{equivrm}
\Psi \colon C \in \F_{q^2}^{n\times n} \mapsto AC^\sigma B+R \in \F_{q^2}^{n\times n}
\end{equation}
or
\[ \Psi \colon C \in \F_{q^2}^{n\times n} \mapsto A(C^\sigma)^T B+R \in \F_{q^2}^{n\times n},\]
where $A,B \in \mathrm{GL}(n,q^2)$, $\sigma \in \mathrm{Aut}(\F_{q^2})$, $R\in \F_{q^2}^{n\times n}$ and $C^T$ denotes the transpose of $C$.
For two subset $\mathrm{C}_1$ and $\mathrm{C}_2$ of ${\rm H}_{n}(q^2)$, if there exists $\Psi$ as above such that
\[ \mathrm{C}_2=\{\Psi(C) \colon C \in \mathrm{C}_1\} \]
we say that $\mathrm{C}_1$ and $\mathrm{C}_2$ are said \emph{extended equivalent}.
Clearly, if $\mathrm{C}_1$ and $\mathrm{C}_2$ of ${\rm H}_{n}(q^2)$ are equivalent in ${\rm H}_{n}(q^2)$, they are also extended equivalent.
However, when maximum $d$-codes are considered, the converse statement is not true. In fact, from what Yue Zhou points out in \cite{Zhou}, it follows that constructions of commutative {\it semifields} exhibited in \cite{Coulter_Henderson} and in \cite{Zhou_Pott} provide examples of maximum $n$-codes in ${\rm H}_{n}(q^2)$ say $\mathrm{C}$, with the property that there exist $A,B \in \mathrm{GL}(n,q^2)$ such that
$$A \mathrm{C} B \subseteq{\rm H}_{n}(q^2),$$
where $A \neq a  B^*$ for each $a \in \F_{q}$.

Along the lines of what has been done by Zhou in \cite{Zhou}, in Section \ref{sec:equivalence} we will investigate on the conditions that guarantee the identification of the aforementioned types of equivalence for maximum Hermitian $d$-codes.
Results in Section \ref{sec:equivalence} heavily rely on what Schmidt proven in \cite{Schmidt2018} using the machinery of association schemes.
Moreover, in Section \ref{sec:ideal} we will show that providing such conditions hold true for a $d$-code $\mathrm{C} \in {\rm H}_{n}(q^2)$, then its {\it idealisers} are both isomorphic to $\F_{q^2}$, and hence they cannot be used as {\it distinguisher}, similarly to what happens in the symmetric setting as proved in \cite{Zhou}.

\medskip
\noindent In Section \ref{sec:setting}, following \cite{LLTZ}, we introduce the Hermitian setting from a polynomial point of view, where some properties are easier to establish. Indeed, we show some combinatorial properties of the known constructions of maximum Hermitian codes.
Finally, in Section \ref{sec:construction} we extend the construction presented in \cite{LLTZ} yielding an example of maximum Hermitian $2$-code and, relying on the results of the previous sections, we are able to show that it is also new.

\section{The association scheme of Hermitian matrices}

By \cite[Section 9.5]{BrouwerCohen} we have that ${\rm H}_n(q^2)$ gives rise to an association scheme whose classes are
\[ (A,B)\in R_i \Leftrightarrow \rk(A-B)=i. \]
Let $\chi \colon \F_q \rightarrow \mathbb{C}$ be a nontrivial character of $(\F_q,+)$ and let
\[ \la A,B \ra=\chi(\mathrm{tr}(A^*B)), \]
with $A,B \in {\rm H}_n(q^2)$ and $\mathrm{tr}$ denotes the matrix trace.
Denoting by $\rm{H}_i$ the subset of ${\rm H}_n(q^2)$ of matrices having rank equal to $i$, the \emph{eigenvalues} of such association scheme are
\[ Q_k(i)=\sum_{A \in \cH_k} \la A,B\ra,\,\,\,\text{for}\,\, B \in \cH_i, \]
with $i,k \in \{0,1,\ldots,n\}$, see \cite{CarlitzHodges,Schmidt2018,Stanton}.

Let $\mathrm{C} \subseteq {\rm H}_n(q^2)$. The \emph{inner distribution} of $\mathrm{C}$ is $(A_0,A_1,\ldots,A_n)$ of rational numbers given by
\[ A_i=\frac{|(\mathrm{C}\times \mathrm{C})\cap R_i|}{|\mathrm{C}|}. \]
Therefore, $\mathrm{C}$ is a $d$-code if and only if
\[ A_1=\ldots=A_{d-1}=0. \]
The \emph{dual inner distribution} of $\mathrm{C}$ is $(A_0',A_1',\ldots,A_n')$ where
\[ A_k'=\sum_{i=0}^n Q_k(i) A_i. \]
Also, we have that $A_0'=|\mathrm{C}|$,  $A_k'\geq 0$ for each $k \in \{0,1,\ldots,n\}$ and if $\mathrm{C}$ is additive then $|\mathrm{C}|$ divides $A_i'$ for each $i\in \{0,\ldots,n\}$.

If $A_1'=\ldots=A_t'=0$, we say that $\mathrm{C}$ is a $t$-\emph{design}. Of course, if $\mathrm{C}$ is additive the $A_i$'s count the number of matrices in $\mathrm{C}$ of rank $i$ with $i\in \{0,1,...,n\}$.

Moreover, in such a case we can associate with $\mathrm{C}$ its dual in  ${\rm H}_n(q^2)$; i.e.,
$$\mathrm{C}^{\perp}=\{X \in {\rm H}_n(q^2) \, \colon \, \langle X,Y\rangle=1 \, \text{ for each } \, Y \in \mathrm{C} \},$$ and it is possible to show that the coefficients $\frac{A_k'}{|\mathrm{C}|}$ count exactly the number of matrices in $\mathrm{C}^{\perp}$ of rank $i$ with $i\in \{0,1,...,n\}$.

Also in \cite{Schmidt2018} the author proved the following results on combinatorial properties of maximum additive Hermitian $d$-codes when $d$ is odd.

\begin{theorem}\cite[Theorem 1]{Schmidt2018}\label{th:maxdodd}
If $\mathrm{C} \subseteq {\rm H}_{n}(q^2)$ is a Hermitian additive $d$-code with odd $d$, then it is maximum if and only if $\mathrm{C}$ is an $(n-d+1)$-design.
\end{theorem}

Consider $m$ and $\ell$ two non-negative integers,  \emph{negative $q$-binomial coefficient} is defined as

\[ \qbin{m}{\ell}{}= \prod_{i=1}^\ell \frac{(-q)^{m-i+1}-1}{(-q)^i-1}. \]

We will need the following property for negative $q$-binomial coefficients. Let $k$ and $i$ be two non-negative integers, then

\begin{equation}\label{eq:propqbin}
    \sum_{j=i}^k (-1)^{j-i} (-q)^{\binom{j-i}{2}} \qbin{j}{i}{}\qbin{k}{j}{}=\delta_{k,i},
\end{equation}

where $\delta_{k,i}$ is the Kronecker delta function, see \cite[Equation (6)]{Schmidt2018} and \cite[Equation (10)]{DelsarteGoethals}.

If $\mathrm{C}$ is a Hermitian additive $d$-code and a $(n-d)$-design, then its inner distribution has beeen determined.

\begin{theorem}\cite[Theorem 3]{Schmidt2018}\label{th:theorem3}
If $\mathrm{C}$ is a Hermitian additive $d$-code and a $(n-d)$-design, then
\[ A_{n-i}=\sum_{j=i}^{n-d} (-1)^{j-i}(-q)^{\binom{j-i}{2}} \qbin{j}{i}{}\qbin{n}{j}{}\left( \frac{|\mathrm{C}|}{q^{nj}}(-1)^{(n+1)j}-1 \right), \]
for each $i \in \{0,1,\ldots,n-1\}$.
\end{theorem}

\section{The equivalence issue for maximum codes}\label{sec:equivalence}

Following the paper of Zhou \cite{Zhou}, we may generalize his considerations to the Hermitian setting.

Let $\mathrm{C}$ be a subset of $\F_{q^2}^{n\times n}$ and let $\mathbf{0}$ be the zero vector in $\F_{q^2}^n$. In \cite{LTZ2} the authors define the following incidence structure

\[ S(\infty)=\{(\mathbf{0},\mathbf{y}) \colon \mathbf{y}\in \F_{q^2}^n\}, \]
\[ S(X)=\{(\mathbf{x},\mathbf{x}X) \colon \mathbf{x}\in \F_{q^2}^n\}, \,\, \text{for}\,\, X \in \mathrm{C}. \]
The \emph{kernel} $K(\mathrm{C})$ of $\mathrm{C}$ is defined as the set of all the endomorphism $\mu$ of the group $(\F_{q^2}^{2n},+)$ such that $S(X)^\mu\subseteq S(X)$ for every $X \in \mathrm{C} \cup \{\infty\}$.
The following result has been proved in \cite{LTZ2}.

\begin{lemma}\label{lemma:LTZ2}
Let $\mathrm{C}$ be a subset of $\F_{q^2}^{n\times n}$.
\begin{enumerate}
  \item [(a)] The kernel of $\mathrm{C}$ is a ring under addition and composition of maps.
  \item [(b)] If $\mathrm{C}_1$ and $\mathrm{C}_2$ are two equivalent rank metric codes in $\F_{q^2}^{n\times n}$, then their kernels are equivalent in $\F_{q^2}^{n\times n}$.
  \item [(c)] Let $I_{n}$ denote the identity matrix of $\F_{q^2}^{n\times n}$. The set of matrices $\{aI_{n+n} \colon a \in \F_{q^2}\}$ forms a field isomorphic to $\F_{q^2}$ contained in $K(\mathrm{C})$.
  \item [(d)] Let $O$ be the zero matrix in $\F_{q^{2}}^{n\times n}$. If $O\in \mathrm{C}$, then each element of $K(\mathrm{C})$ must be of the form
      \begin{equation}\label{eq:kernel}
      \left(\begin{array}{cccc} N_1 & O \\ O & N_2 \end{array} \right),
      \end{equation}
       where $N_1,N_2 \in \mathrm{End}(\mathbb{F}_{q^2}^n,+)$.
\end{enumerate}
\end{lemma}

As a consequence we can prove the following result.

\begin{lemma}\label{lemma:kernelfield}
Let $\mathrm{C}$ be a subset of ${\rm H}_n(q^2)$ containing $O$ and $I_n$.
If there are no trivial subspaces $U$ and $W$ such that
\begin{itemize}
  \item $\F_{q^2}^n=U\oplus W$;
  \item $\{\mathbf{u}X \colon \mathbf{u} \in U, X \in \mathrm{C}\}\subseteq U$;
  \item $\{\mathbf{w}X \colon \mathbf{w} \in W, X \in \mathrm{C}\}\subseteq W$,
\end{itemize}
then the kernel of $\mathrm{C}$ is isomorphic to a finite field containing $\F_{q^2}$.
\end{lemma}
\begin{proof}
Since $O\in \mathrm{C}$, by (d) of Lemma \ref{lemma:LTZ2} each element $A$ of $K(\mathrm{C})$ is of Form \eqref{eq:kernel}, i.e.
\[ A=\left(\begin{array}{cccc} N_1 & O \\ O & N_2 \end{array} \right). \]
Because of (a) of Lemma \ref{lemma:LTZ2}, it is enough to show that except for the case in which $N_1$ and $N_2$ are the zero matrix, $N_1$ and $N_2$ are invertible.
Since $A \in K(\mathrm{C})$, then
\[ \{ (\mathbf{x}N_1,\mathbf{x}XN_2)\colon \mathbf{x} \in \F_{q^2}^n \}\subseteq\{ (\mathbf{x},\mathbf{x}X)\colon \mathbf{x} \in \F_{q^2}^n \}, \]
and hence $\mathbf{x}N_1X=\mathbf{x}XN_2$ for each $\mathbf{x}\in \F_{q^2}^n$.
Since $I_n \in \mathrm{C}$, we may choose $X=I_n$ and hence we have $N_1=N_2$, which will be denoted by $N$.
Suppose that $\mathbf{x}N=\mathbf{0}$, then we have also that $\mathbf{x}XN=\mathbf{0}$.
This implies that each $X\in \mathrm{C}$ maps the kernel of $N$ into itself.
Denote by $V$ the kernel of $N$ and by $k$ its dimension.
Choosing a suitable basis of $\F_{q^2}^n$ in such a way that its first $k$ elements are a basis of $V$, then each element of $\mathrm{C}$ may be written as
\[ \left( \begin{array}{cccc} X_1 & O \\ O & X_2 \end{array}\right), \]
with $X_1 \in \mathrm{H}_k(q^2)$ and $X_2\in \mathrm{H}_{n-k}(q^2)$.
Let $U$ and $W$ be the subspaces corresponding to the first $k$ coordinates and the last $n-k$ coordinates respectively.
If $k>0$ this would contradict the hypothesis and hence $N_1$ and $N_2$ are invertible.
\end{proof}

\subsection{The equivalence issue}

In this section we will show that, under some assumptions, the equivalence of two maximum additive hermitian $d$-codes in ${\rm H}_n(q^2)$ coincides with extended equivalence in $\F_{q^2}^{n\times n}$.

\begin{theorem}\label{th:kernel}
Let $d$ be a positive integer and let $\mathrm{C}$ be a maximum additive $d$-code in ${\rm H}_n(q^2)$.
If there exist $a \in \F_{q}^*$ and $P \in \mathrm{GL}(n,q^2)$ such that
\[ I_n \in aP^*X P, \]
then $K(\mathrm{C})$ is isomorphic to a finite field containing $\F_{q^2}$. In particular, if $d<n$ then $K(\mathrm{C})$ is isomorphic to $\F_{q^2}$.
\end{theorem}
\begin{proof}
Clearly, by (b) Lemma \ref{lemma:LTZ2}, we may assume that $I_n \in \mathrm{C}$.
Now, we show that the hypothesis in Lemma \ref{lemma:kernelfield} are satisfied and hence $K(\mathrm{C})$ is a finite field.
Suppose that there exist two subspaces $U$ and $W$ of $\F_{q^2}^n$ such that $\F_{q^2}^n=U\oplus W$ and
 \begin{itemize}
  \item $\{\mathbf{u}X \colon \mathbf{u} \in U, X \in \mathrm{C}\}\subseteq U$ and
  \item $\{\mathbf{w}X \colon \mathbf{w} \in W, X \in \mathrm{C}\}\subseteq W$.
\end{itemize}
Let $k$ be the dimension of $U$ and we may assume that $k\geq\lfloor\frac{n}2\rfloor$ and that a basis for $U$ is given by the first $k$ elements of the standard basis of $\F_{q^2}^n$.
Therefore, each element $M$ of $\mathrm{C}$ can be written as
\[ M=\left(\begin{array}{cccc} M_1 & O \\ O & M_2 \end{array}\right), \]
with $M_1\in {\rm H}_n(q^2)$ and $M_2 \in {\rm H}_{n-k}(q^2)$.
 \begin{itemize}
   \item If $d>\lfloor\frac{n}2\rfloor$, then the set
        \[ \mathrm{C}_1:=\{M_1 \colon M \in \mathrm{C}\} \]
        has size $|\mathrm{C}|=q^{n(n-d+1)}$, otherwise there would be two matrices in $\mathrm{C}$ whose difference has rank less than or equal to $n-k\leq \lfloor\frac{n}2\rfloor$. Its minimum distance $d_1$ is greater than or equal to $d-(n-k)$.
        Bound \eqref{eq:SingletonBoundHermitian} applied to $\mathrm{C}_1$ implies
        \[ q^{n(n-d+1)}=|\mathrm{C}_1|\leq q^{k(k-d_1+1)}\leq q^{k(k-d+(n-k)+1)}. \]
        Thus $k=n$.
   \item Suppose that $d\leq \lfloor\frac{n}2\rfloor$. For each $M_2 \in {\rm H}_{n-k}(q^2)$ let
        \[ \mathrm{C}_{M_2}=\left\{ M_1 \colon \left(\begin{array}{cccc} M_1 & O \\ O & M_2 \end{array}\right) \in \mathrm{C} \right\}. \]
        Its minimum distance $d(\mathrm{C}_{M_2})\geq d$ and by \eqref{eq:SingletonBoundHermitian},
        \[ |\mathrm{C}_{M_2}| \leq q^{k(k-d+1)}. \]
        Therefore,
        \[ |\mathrm{C}|=\sum_{M_2 \in \mathrm{H}_{n-k}(q^2)} |\mathrm{C}_{M_2}|\leq q^{(n-k)(n-k+1)}\cdot q^{k(k-d+1)}, \]
        and so
        \[ n(n-d+1)\leq (n-k)^2+(n-k)+k(k-d+1)\leq (n-k)^2+(n-k)+k(n-d+1).\]
        If $k\neq n$ then $d \geq k$, which is not possible.
        Hence $k=n$.
 \end{itemize}
In both the aforemetioned cases we have $k=n$ and therefore we can apply Lemma \ref{lemma:kernelfield} and (c) of Lemma \ref{lemma:LTZ2} to get the first part of the assertion.
Now, suppose that $d<n$ and that $K(\mathrm{C})\simeq \F_{q^{2\ell}}$ contains properly a field isomorphic to $\F_{q^2}$.
Then $\mathrm{C}$ can be seen as subset of Hermitian matrices of order $n/\ell$ over $\mathbb{F}_{q^{2\ell}}$ with minimum distance $d'=d/\ell$.
By \eqref{eq:SingletonBoundHermitian} we have that
\[ |\mathrm{C}|=q^{n(n-d+1)}\leq q^{\frac{n}{\ell}\left(\frac{n}{\ell}-d'+1\right)}, \]
from which we get $\ell=1$ and also the second part of the statement follows.
\end{proof}

\begin{lemma}\label{lemma:invmat}
If $\mathrm{C}$ is a Hermitian maximum additive $d$-code and an $(n-d)$-design with $d<n$.
Then there is at least one invertible matrix in $\mathrm{C}$.
\end{lemma}
\begin{proof}
If $d=1$, then $\mathrm{C}={\rm H_n(q^2)}$ and the assertion holds.
So assume that $1<d<n$: our aim is to prove that $A_n \neq 0$.
By Theorem \ref{th:theorem3}, we have that
\[ A_{n-i}=\sum_{j=i}^{n-d} (-1)^{j-i}(-q)^{\binom{j-i}{2}} \qbin{j}{i}{}\qbin{n}{j}{}\left( \frac{|\mathrm{C}|}{q^{nj}}(-1)^{(n+1)j}-1 \right), \]
for each $i \in \{0,1,\ldots,n-1\}$.
For $i=0$, we get
\begin{equation}\label{eq:An}
A_{n}=\sum_{j=0}^{n-d} (-1)^{j}(-q)^{\binom{j}{2}} \qbin{j}{0}{}\qbin{n}{j}{}\left( \frac{|\mathrm{C}|}{q^{nj}}(-1)^{(n+1)j}-1 \right).
\end{equation}
Recalling that $|\mathrm{C}|=q^{n(n-d+1)}$, the above formula can be written as follows
\[A_{n}=\sum_{j=0}^{n-d} (-1)^{j}(-q)^{\binom{j}{2}} \qbin{n}{j}{}\left(q^{n(n-d-j+1)}-1 \right)\]
\[\equiv -\sum_{j=0}^{n-d} (-1)^{j}(-q)^{\binom{j}{2}} \qbin{n}{j}{} \pmod{q^{n-d}} \]
\[\equiv -\sum_{j=0}^{n} (-1)^{j}(-q)^{\binom{j}{2}} \qbin{n}{j}{} +\sum_{j=n-d+1}^{n} (-1)^{j}(-q)^{\binom{j}{2}} \qbin{n}{j}{} \pmod{q^{n-d}} \]
\[ \equiv -\sum_{j=0}^{n} (-1)^{j}(-q)^{\binom{j}{2}} \qbin{n}{j}{} \pmod{q^{n-d}}.  \]
Therefore, by Equation \eqref{eq:propqbin} we have $A_n \equiv -1 \pmod{q^{n-d}}$, so that $A_n\ne 0$.
\end{proof}

We are ready to prove the main result of this section.

\begin{theorem}\label{th:equiv}
If $\mathrm{C}_1$ and $\mathrm{C}_2$ are two maximum additive Hermitian $d$-codes and $(n-d)$-designs with $d<n$.
Then they are equivalent in ${\rm H}_n(q^2)$ if and only if they are extended equivalent.
\end{theorem}
\begin{proof}
Clearly, if $\mathrm{C}_1$ and $\mathrm{C}_2$ are equivalent in ${\rm H}_n(q^2)$ then they are also extended equivalent.
Now assume that $\mathrm{C}_1$ and $\mathrm{C}_2$ are extended equivalent, i.e. there exist two invertible matrices $A,B \in \mathrm{GL}(n,q^2)$, $\rho \in \mathrm{Aut}(\F_{q^2})$ and $R \in \F_{q^2}^{n\times n}$ such that
\[ \mathrm{C}_1=A \mathrm{C}_2^\rho B +R. \]
Since $\mathrm{C}_1$ and $\mathrm{C}_2$ are additive, we may assume that $R=O$, i.e. $\mathrm{C}_1=A \mathrm{C}_2^\rho B$.
We are going to prove that $A=zB^*$ for some $z \in \F_{q}^*$.
So,
\[ \mathrm{C}_2=A \mathrm{C}_1^\sigma B=(A(B^*)^{-1})B^*\mathrm{C}_1^\sigma B=M\mathrm{C}_3, \]
where $M=A(B^*)^{-1}$ and $\mathrm{C}_3=B^*\C_1^\sigma B \subseteq {\rm H}_n(q^2)$.
As a consequence, we have that $MX \in {\rm H}_n(q^2)$ for each $X \in \mathrm{C}_3$, i.e.
\[ MX=(MX)^*=XM^* \]
for all $X \in \mathrm{C}_3$.
Hence the matrix
\[ \left( \begin{array}{cccc} M & O \\ O & M^* \end{array} \right) \in K(\mathrm{C}_3). \]
By Lemma \ref{lemma:invmat}, there exists in $\mathrm{C}_3$ an invertible matrix, which implies the existence of $a \in \F_{q}$ and $D \in \mathrm{GL}(n,q)$ such that $I_n\in a D^*\mathrm{C}_3 D$.
Now, by Theorem \ref{th:kernel} we have that $K(\mathrm{C}_3)=\F_{q^2}$ and hence $M=z I_n$ for some $z \in \F_{q^2}^*$. By (c) of Lemma \ref{lemma:LTZ2}, we have
\[ K(\mathrm{C}_3)=\{\gamma I_{n+n} \colon \gamma \in \F_{q^2}\}, \]
and as $\left( \begin{array}{cccc} M & O \\ O & M^* \end{array} \right) \in K(\mathrm{C}_3)$, it follows that $M=M^*=z I_n$, with $z \in \F_{q}^*$,
i.e. $A=zB^*$.
\end{proof}

As a consequence of Theorem \ref{th:maxdodd} we get the following.

\begin{corollary}\label{cor:equiv}
If $\mathrm{C}_1$ and $\mathrm{C}_2$ are two Hermitian maximum additive $d$-codes with $d$ odd, $d<n$.
Then they are equivalent in ${\rm H}_n(q^2)$ if and only if they are extended equivalent.
\end{corollary}

\section{Idealisers are not distinguishers in ${\rm H}_n(q^2)$}\label{sec:ideal}

In the classical rank metric context, to establish whether two codes are equivalent or not could be quite difficult.
One of the strongest tool for such a issue is given by the automorphism groups of such codes, which usually is very hard to determine.
In some cases it is enough to study some subgroups of the automorphism group which are invariant under the equivalence, which are easier to calculate, such as the \emph{idealisers} introduced in \cite{LN2016} and deeply investigated in \cite{LTZ2}.

Let $\mathrm{C}$ be an additive rank metric code in $\F_{q}^{n\times n}$, its \emph{left idealiser} $I_\ell(\mathrm{C})$ is defined as
\[ I_\ell(\mathrm{C})=\{Z \in \F_{q}^{n\times n} \colon ZX \in \mathrm{C} \,\,\text{for all}\,\, X \in \mathrm{C}\}  \]
and its \emph{right idealiser} $I_r(\mathrm{C})$ is defined as
\[ I_r(\mathrm{C})=\{Z \in \F_{q}^{n\times n} \colon XZ \in \mathrm{C} \,\,\text{for all}\,\, X \in \mathrm{C}\}.  \]

Idealisers have been used to \emph{distinguish} examples of MRD-codes, see \cite{BZZ,CMPZ,CsMPZh,CsMZ2018,LLTZ,LTZ2,MMZ,SchmidtZhou,ZZ}.
In the next we prove that for maximum additive Hermitian $d$-codes left and right idealisers are isomorphic to $\F_{q}$, i.e. they cannot be used as distinguishers in the Hermitian setting.

\begin{theorem}
Let $\mathrm{C}$ be a maximum Hermitian additive $d$-code and a $(n-d)$-design with $d<n$.
Then $I_\ell(\mathrm{C})$ and $I_r(\mathrm{C})$ are both isomorphic to $\F_{q}$.
\end{theorem}
\begin{proof}
Let us consider the left idealiser case and let $M \in I_\ell(\mathrm{C})$.
We have that $MX \in {\rm H}_n(q^2)$ for each $X \in \mathrm{C}$, i.e.
\[ MX=(MX)^*=XM^* \]
for all $X \in \mathrm{C}$.
Hence the matrix
\[ \left( \begin{array}{cccc} M & O \\ O & M^* \end{array} \right) \in K(\mathrm{C}), \]
and as in the proof of Theorem \ref{th:equiv}, we get that $M=a I_n$ for some $a \in \F_{q}$.
Similar arguments imply the same result for the right idealiser.
\end{proof}

As a consequence of Theorem \ref{th:maxdodd} we get the following.

\begin{corollary}
If $\mathrm{C}$ is a maximum Hermitian additive $d$-code with $d$ odd, $d<n$.
Then $I_\ell(\mathrm{C})$ and $I_r(\mathrm{C})$ are both isomorphic to $\F_{q}$.
\end{corollary}

\section{The $q$-polynomial setting and some combinatorial properties}\label{sec:setting}

We briefly introduce the Hermitian setting from a polynomial point of view.
Let $n \in \mathbb{Z}^+$ be a positive integer, and let $q$ be a prime power.
We denote by $\mathcal{L}_{n,q}$ the quotient $\F_q$-algebra of the algebra of linearized polynomials over $\F_{q^n}$ with respect to $(x-x^{q^n})$, i.e.
\[\mathcal{L}_{n,q}=\left\{ \sum_{i=0}^{n-1} a_i x^{q^i} \colon a_i \in \F_{q^n} \right\}. \]

It is well known that there is a {\it one-to-one} correspondence between the elements of $\mathcal{L}_{n,q}$ and the $\mathbb{F}_q$-linear transformation of $\F_{q^n}$ (represented as matrices).  Using this fact and following the point of view expressed in \cite{LLTZ}, we may identify the set ${\rm H}_n(q^2)$ of Hermitian matrices of order $n$ over $\F_{q^{2}}$ with the set of $q^2$-polynomials
\[{\cH}_n(q^2)=\left\{ \sum_{i=0}^{n-1} c_i x^{q^{2i}} \colon c_{n-i+1}=c_i^{q^{2n-2i+1}}, \,\text{ with } \, i \in \{0,\ldots,n-1\} \right\}\subseteq \mathcal{L}_{n,q^2}, \]
where the indices are taken modulo $n$. We underline here that if $n$ is odd then $c_{(n+1)/2} \in \F_{q^n}$. Moreover, the rank of a Hermitian form equals the dimension of the image of the map $f  \, : \, \F_{q^{2n}} \rightarrow \F_{q^{2n}}$, where $f \in {\cH}_n(q^2)$.

Also, we may consider the maps that preserve the rank distance in ${\rm H}_{n}(q^2)$ represented as polynomials.
In order to do this, consider the non-degenerate symmetric bilinear form of $\F_{q^{2n}}$ over $\F_{q^2}$ defined by
\[ \la x,y\ra= \Tr_{q^{2n}/q^2}(xy), \]
for each $x,y \in \F_{q^{2n}}$, where $\displaystyle \Tr_{q^{2n}/q^2}(x)=\sum_{i=0}^{n-1}x^{q^{2i}}$. Then the \emph{adjoint} $f^\top$ of the linearized polynomial $\displaystyle f(x)=\sum_{i=0}^{n-1} a_ix^{q^{2i}} \in \mathcal{L}_{n,q^2}$ with respect to the bilinear form $\la,\ra$ is
\[ f^\top(x)=\sum_{i=0}^{n-1} a_i^{q^{n-2i}}x^{q^{n-2i}}, \]
i.e.
\[ \Tr_{q^{2n}/q^2}(xf(y))=\Tr_{q^{2n}/q^2}(y{f}^\top(x)), \]
for any $x,y \in \F_{q^{2n}}$.

Then, one can easily verify that maps preserving the rank distance in ${\rm H}_{n}(q^2)$, are of the form
\begin{equation}\label{eq:def_equivalence_hermitian_setting}
	\Theta_{a,g,\rho,r_0}(f) = ag \circ f^{\rho} \circ g^{{\top}q^{2n-1}}(x) + r_0(x),
\end{equation}
for given $a \in \F^*_{q}$, $\rho \in \Aut(\F_{q^2})$, $g(x)=\sum_{i=0}^{n-1}g_ix^{q^i}$ a permutation $q^2$-polynomial over $\F_{q^{2n}}$, $r_0 \in {\cH}_{n}(q^2)$ and $g^{{\top}q^{2n-1}}(x)=\sum_{i=0}^{n-1} g_i^{q^{n-2i-1}}x^{q^{n-2i}}$.

In this context, if $\cC_1$ and $\cC_2$ are two subsets of ${\cH}_{n}(q^2)$ and there exists a map $\Theta_{a,g,\rho,r_0}$ defined as in Equation \eqref{eq:def_equivalence_hermitian_setting} for certain $a$, $g$, $\rho$ and $r_0$ such that
\[ \C_2 :=\{\Theta_{a,g,\rho,r_0}(f): f\in \C_1 \},\]
then we say that $\C_1$ and $\C_2$ are \emph{equivalent} in ${\cH}_{n}(q^2)$.

As we are considering $d$-codes using linearized polynomials, we can interpret the dual code $\cC^{\perp}$ of $\cC$ in the following way:

\[\cC^{\perp}=\{ f \in \cH_n(q^2) \, \colon \, b(f,g)=0, \,\, \forall \,\, g \in \cC\},\] where

\begin{equation}\label{formula:q-polybilinearform} b(f,g)=\mathrm{Tr}_{q^{2n}/q^2}\left(\sum_{i=0}^{n-1} a_ib_i\right),\end{equation}
whenever $f(x)=\sum_{i=0}^{n-1}a_i x^{q^{2i}}$ and $g=\sum_{i=0}^{n-1}b_i x^{q^{2i}} \in \cH_{n}(q^2)$.

\begin{remark}
As noted in \cite[Section 2]{Sheekey2016} (see also \cite{LMPTSymp}), there exists an $\mathbb{F}_{q^2}$-basis of $\mathbb{F}_{q^2}^n$ such that $\mathrm{H}_n(q^2)$ and $\mathcal{H}_n(q^2)$ are isomorphic (denote by $\varphi$ such an isomorphism) and with the property that $\mathrm{tr}(A^*B)=b(\varphi(A),\varphi(B))$.
Now, recalling that $\langle A, B\rangle =1$ if and only if $b(\varphi(A),\varphi(B))=0$ (as $\chi$ is a non-trivial character of $\mathbb{F}_{q}$), we have that
\[ \varphi(\mathrm{C}^\perp)=\mathcal{C}^\perp. \]
This allows us to switch between the two models.
\end{remark}

 \medskip

\noindent Here below we give a description of the known examples of maximum Hermitian $d$-codes in a polynomial fashion, \cite[Theorems 4 and 5]{Schmidt2018} (see also \cite[Section 2.2]{LLTZ}). More precisely, let $s$ be an odd positive integer with $\gcd(s,n)=1$.
If $n$ and $d$ are integers with opposite parity such that $1\leq d \leq n-1$, then the set

\begin{equation}\label{ex:H}
\mathcal{H}_{n,d,s}=\left\{ \sum_{j=1}^{\frac{n-d+1}2} \left( (b_jx)^{q^{2s(n-j+1)}}+b_j^{q^s} x^{q^{2sj}} \right)  \colon b_1,\ldots, b_{\frac{n-d+1}2}\in \F_{q^{2n}} \right\},
\end{equation}

\noindent is a maximum $\F_q$-linear Hermitian $d$-code.

In addition, if $n$ and $d$ are both odd integers, then the set

\begin{equation}\label{ex:E}
\mathcal{E}_{n,d,s}=\left\{ (b_0x)^{q^{s(n+1)}}+\sum_{j=1}^{\frac{n-d}2} \left( (b_jx)^{q^{s(n+2j+1)}}+b_j^{q^s} x^{q^{s(n-2j+1)}} \right) \colon b_0\in \F_{q^n}, b_1,\ldots, b_{\frac{n-d+1}2}\in \F_{q^{2n}} \right\},
\end{equation}
is a maximum $\F_q$-linear Hermitian $d$-code.

We present some combinatorial properties of these examples. In order to do this, let us recall the following result of Gow and Quinlan.

\begin{theorem}{(\cite[Theorem 5]{GQ2009} and \cite[Theorem 10]{GQ2009x})}\label{GQ}
The dimension of the kernel of a $q$-polynomial $f(x)=a_0x+a_1x^{q^s}+\cdots+a_{k-1}x^{q^{s(k-1)}}+a_k x^{q^{sk}}\in \mathcal{L}_{n,q}$ with $\gcd(s,n)=1$ is at most $k$.
In particular, if the dimension of the kernel of $f(x)$ is $k$, then $\N_{q^n/q}(a_0)=(-1)^{nk}\N_{q^n/q}(a_k)$, where $\N_{q^n/q}(a)=a^{\frac{q^n-1}{q-1}}$ for $a \in \F_{q^n}$.
\end{theorem}

The next result provides combinatorial properties of Constructions \eqref{ex:H} and \eqref{ex:E}.

\begin{theorem}\label{th:desiother}
For any suitable parameters $n,d$ and $s$, the maximum $\F_q$-linear $d$-codes $\mathcal{H}_{n,d,s}$ and $\mathcal{E}_{n,d,s}$ are $(n-d+1)$-designs.
\end{theorem}
\begin{proof}
If $d$ is odd, the assertion follows by Theorem \ref{th:maxdodd}.
So, the remaining codes to be analyzed are $\mathcal{H}_{n,d,s}$ with $n$ odd and $d$ even.
Let start by determining its dual code $\mathcal{H}_{n,d,s}^{\perp}$ with respect to the bilinear form (\ref{formula:q-polybilinearform}).
First, we remark that
\begin{equation}\label{eq:card} |\mathcal{H}_{n,d,s}^\perp|= \frac{q^{n^2}}{|\mathcal{H}_{n,d,s}|}=q^{n(d-1)}. \end{equation}
Let us consider the following set
\[ \mathcal{D}:=\left\{ c_{\frac{n+1}2} x^{q^{2s\frac{n+1}{2}}}+ \sum_{i=\frac{n-d+3}2}^{\frac{n-1}2} c_i x^{q^{2si}}+c_i^{q^{2n-2i+1}}x^{q^{2s(n-i+1)}} \colon c_{\frac{n+1}2} \in \F_{q^n},\right.\]
\[\left. c_i \in \F_{q^{2n}}, i \in \left\{\frac{n-d+3}2, \ldots,\frac{n-1}2 \right\} \right\}. \]
It follows that each polynomial $f$ in $\mathcal{D}$ satisfies the property that \[b(f,h)=0 \,\text{ for any }\,  h \in \mathcal{H}_{n,d,s}.\] Hence, by \eqref{eq:card} we have that $\mathcal{D}=\mathcal{H}_{n,d,s}^\perp$.
Let us consider
\[ \mathcal{D}\circ x^{q^{2s(n-\frac{n-d+3}{2})}}=\{f \circ x^{q^{2s(n-\frac{n-d+3}{2})}} \colon f(x) \in \mathcal{D} \}. \]
The polynomials in $\mathcal{D}\circ x^{q^{2s(n-\frac{n-d+3}{2})}}$ have $q^{2s}$-degree less than or equal to $d-1$, and hence by Theorem \ref{GQ} we have that
\[ \dim_{\F_{q^2}} \ker f(x)= \dim_{\F_{q^2}} \ker f \circ x^{q^{2s(n-\frac{n-d+3}{2})}} \leq d-1,   \]
for each $f \in \mathcal{D}\setminus\{0\}$, i.e. $\mathrm{rk}\, f\geq n-d+1$ for each $f \in \mathcal{D}\setminus\{0\}$.
Hence $\mathcal{D}$ is an $(n-d+1)$-code and the assertion is then proved.
\end{proof}

Moreover in \cite{MiriamSchmidt_master'sthesis} and in \cite{Schmidt2018} another family of additive $2$-codes in $\mathrm{H}_n(q^2)$ was exhibited which exists for any value of the positive integer $n$. In fact, \begin{equation}\label{example:Myriam}\mathrm{M} = \{(m_{i,j})_{1\leq i,j \leq n} \, \in \,{\rm H}_n(q^2) \,\, \colon \,\, m_{i,i}=0 \,\,\,\, \forall \,\, 1 \leq i \leq n\},\end{equation} see \cite[Theorem 6.1]{MiriamSchmidt_master'sthesis}. We are going to show that this example is not a $1$-design and hence it cannot be equivalent to the aforementioned families.

\smallskip

By simply adapting arguments exhibited in \cite[Section $3.4$]{Schmidt2010}, designs in the Hermitian association scheme can be characterized by means of the following property

\begin{theorem}\label{th:designsinhermitianscheme}
Let $U$ be a $t$-dimensional vector subspace of $V(n,q^2)=\F_{q^2}^n$ and let $H:\,U\times U \rightarrow \F_{q^2}$ be a Hermitian bilinear form on $U$.
Then, a $d$-code $\mathrm{C} \subset {\rm H}_n(q^2)$ is a $t$-design if and only if the number of forms in $\mathrm{C}$ that are an extension of $H$, is independent of the choice of $U$ and $H$.
\end{theorem}

As a consequence we have the following result.

\begin{theorem}\label{th:nodesiMir}
The $2$-code $\mathrm{M}$ is not a $t$-design for any $t\neq 0$.
\end{theorem}
\begin{proof}
It is enough to show that $\mathrm{M}$ is not a $1$-design.
Indeed, let $U=\langle (1,0,\ldots,0)\rangle_{\F_{q^2}}$ a one-dimensional subspace of $\F_{q^2}^n$.
The number of forms in $\mathrm{M}$ that are extension of the $1\times 1$ Hermitian bilinear for $H=(0)$ is $|\mathrm{M}|$, and the number of forms in $\mathrm{M}$ that are extension of the $1\times 1$ Hermitian bilinear for $H=(1)$ is $0$.
Therefore, by Theorem \ref{th:designsinhermitianscheme} we have that $\mathrm{M}$ is not a $1$-design.
\end{proof}

Therefore, we have the following.

\begin{corollary}
The $2$-code $\mathcal{M}$ is not equivalent to $\mathcal{H}_{n,2,s}$, for any $n$ and $s$.
\end{corollary}

As pointed out in Theorem \ref{th:maxdodd}, any maximum $d$-code is an $(n-d+1)$-design when $d$ is odd.
For the $d$ even case this is not true. Indeed, by Theorem \ref{th:nodesiMir}, we have example of maximum $2$-code which is not even a $1$-design, whereas by Theorem \ref{th:desiother} we have examples of maximum $d$-codes which are $(n-d+1)$-designs.

\section{New constructions of maximum Hermitian $2$-code}\label{sec:construction}

We start by pointing out the technique developed in \cite{TZ}, in order to use it in the Hermitian setting similarly to what has been done in \cite{LLTZ} in the symmetric framework.

\smallskip

In \cite{TZ}, the following was proved.

\begin{lemma}\label{tec}
Let $q$ be an odd prime power, let $n \in\mathbb{Z}^+$ and $s \in \mathbb{Z}$ be two integers such that $n$ is odd and $(s,2n)=1$. Let $\gamma \in \F_{q^{2n}}$ with $\N_{q^{2n}/q}(\gamma)$ a non-square in $\F_q$.
If $\displaystyle f(x)=a x+ \sum_{i=0}^{k-1} a_i x^{{q^{is}}}+\gamma b x^{q^{sk}} \in \mathcal{L}_{2n,q}$ with $a_i \in \F_{q^{2n}}$, $a,b \in \F_{q^n}$, then $\dim_{\F_q} \ker f \leq k-1$ and $\mathrm{rk} \,\,f \geq 2n-k+1$.
\end{lemma}
\begin{proof}
By Theorem \ref{GQ} $\dim_{\F_q} \ker f \leq k$. By way of contradiction, let us assume that the dimension of the kernel of $f(x)$ is $k$.
Hence, by Theorem \ref{GQ}, it follows that
\[ \N_{q^{2n}/q}(a)=\N_{q^{2n}/q}(b \gamma), \]
i.e., since $a,b \in \F_{q^n}$,
\[ \N_{q^{2n}/q}(\gamma)=\N_{q^{2n}/q}\left(\frac{a}{b}\right)=\N_{q^{n}/q}\left(\frac{a}{b}\right)^2, \]
which gives a contradiction.
The second part follows from the relation $\mathrm{rk} \,\, f = 2n- \dim_{\F_q} \ker \,\, f$.
\end{proof}

We are now able to generalize the construction of \cite{LLTZ} to the Hermitian setting. Precisely, we have

\begin{theorem}\label{thm:newhermitiancode}
Let $q$ be an odd prime power, let $n \in\mathbb{Z}^+$ and $s \in \mathbb{Z}$ be two integers such that $n$ is odd and $(s,2n)=1$. Let $\gamma \in \F_{q^{2n}}$ with $\N_{q^{2n}/q}(\gamma)$ a non-square in $\F_q$.
Then
\begin{small}
\begin{equation*} \begin{split} \tilde{\mathcal{H}}_s&=\biggl \{ b x^{q^{2s\frac{n+1}{2}}} + a\gamma x^{q^{2s\frac{n-1}{2}}}+(a\gamma)^{q^{s(n+2)}} x^{q^{2s\frac{n+3}{2}}} + \sum_{i=1}^{\frac{n-3}2} (c_ix^{q^{2si}}+c_i^{q^{s(2n-2i+1)}}x^{q^{2s(n-i+1)}}) \\ & \colon a,b \in \F_{q^n}, c_i \in \F_{q^{2n}} \biggr \}. \end{split} \end{equation*}
\end{small}
is a maximum Hermitian $\F_q$-linear $2$-code.
\end{theorem}
\begin{proof}
First we note that $|\tilde{\mathcal{H}}_s|=q^{2n\frac{n-3}2+2n}=q^{n(n-1)}$ which, according to \eqref{eq:SingletonBoundHermitian}, is the maximum possible size providing $d=2$.
Now we have to show that $\dim_{\F_{q^2}} \ker f \leq n-2$ for each $f \in \tilde{\mathcal{H}}_s$.
Indeed, if $\dim_{\F_{q^2}} \ker f \leq n-2$, then $\mathrm{rk}\,\, f \geq n-(n-2)=2$.

By way of contradiction, we may suppose that there exists

\[ f(x)= b x^{q^{2s\frac{n+1}{2}}} + a\gamma x^{q^{2s\frac{n-1}{2}}}+(a\gamma)^{q^{s(n+2)}} x^{q^{2s\frac{n+3}{2}}} + \sum_{i=1}^{\frac{n-3}2} (c_ix^{q^{2si}}+c_i^{q^{s(2n-2i+1)}}x^{q^{2s(n-i+1)}}) \]

\noindent in $\tilde{\mathcal{H}}_s$ such that $\dim_{\F_{q^2}} \ker \,\,f \geq n-1$.
Clearly, the $\dim_{\F_{q^2}} \ker \,\, f = \dim_{\F_{q^2}} \ker \,\, f \circ x^{q^{si}}$ for each $i \in \{0,\ldots,2n-1\}$.
In particular,
\[ f \circ x^{q^{s(n-3)}} := b x^{q^{2s(n-1)}} + a\gamma x^{q^{2s(n-2)}}+(a\gamma)^{q^{s(n+2)}} x + \sum_{i=0}^{\frac{n-3}2} c_ix^{q^{s(2i+n-3)}}+c_i^{q^{2n-2i+1}}x^{q^{s(n-2i-1)}} \]
has $q^{2s}$-degree at most $n-1$ and hence, by Theorem \ref{GQ}, it follows that $\dim_{\F_{q^2}}\ker f \leq n-1$.
When we look at $f\circ  x^{q^{s(n-3)}}$ as a $q$-polynomial in $\F_{q^{2n}}$ we have that $\dim_{\F_q} \ker (f\circ  x^{q^{s(n-3)}})=2n-2$; a contradiction by Lemma \ref{tec}.
Hence, $\dim_{\F_{q^2}} \ker\,\, f \leq n-2$.
\end{proof}



Also we are in the position to determine its dual code $\tilde{\mathcal{H}}^{\perp}_s$ of $\tilde{\mathcal{H}}_s$. Precisely, we have

\begin{theorem}
Let $\gamma \in \F_{q^{2n}}$ with $N_{q^{2n}/q}(\gamma)$ a non-square element of $\F_q$. Then, the dual code of $\tilde{\mathcal{H}}_s$ is
\[ \tilde{\mathcal{H}}_s^\perp=\left\{  c \gamma^{-1} \alpha x^{q^{2s(\frac{n-1}2)}}+(c\gamma^{-1} \alpha)^{q^{s(n+2)}}x^{q^{2s\frac{n+3}2}} \colon c \in \F_{q^n} \right\}, \]
with $\alpha \in \F_{q^{2n}}$ and $\alpha^{q-1}=-1$.
\end{theorem}
\begin{proof}
We have that $|\tilde{\mathcal{H}}_s^{\perp}|=q^{n^2}/|\tilde{\mathcal{H}}_s|=q^n $.  Let
\[f(x)=b x^{q^{2s\frac{n+1}{2}}} + a\gamma x^{q^{2s\frac{n-1}{2}}}+(a\gamma)^{q^{s(n+2)}} x^{q^{2s\frac{n+3}{2}}} + \sum_{i=1}^{\frac{n-3}2} (c_ix^{q^{2si}}+c_i^{q^{s(2n-2i+1)}}x^{q^{2s(n-i+1)}}) \in\tilde{\mathcal{H}}_s\]
and
\[g(x)=c \gamma^{-1} \alpha x^{q^{2s(\frac{n-1}2)}}+(c\gamma^{-1} \alpha)^{q^{s(n+2)}}x^{q^{2s\frac{n+3}2}}\]
with $c \in \F_{q^n}$, then
\[ b(f,g)=\mathrm{Tr}_{q^{2n}/q^2}\left(ac \alpha+ (ac \alpha)^{q^{s(n+2)}}\right)=\mathrm{Tr}_{q^{2n}/q^2}\left(ac \alpha+ ac \alpha^{q}\right)=  \]
\[ \mathrm{Tr}_{q^{2n}/q^2}\left(ac \alpha (1+\alpha^{q-1})\right)=0. \]
The assertion then follows.
\end{proof}

\begin{corollary}\label{cor:designproperty}
The $2$-code $\tilde{\mathcal{H}}_s$ is an $(n-1)$-design.
\end{corollary}
\begin{proof}
To prove the assertion it is enough to show that all the polynomials in $\tilde{\mathcal{H}}_s^\perp$ are invertible.
For this purpose, let $$f(x)= c \gamma^{-1} \alpha x^{q^{2s(\frac{n-1}2)}}+(c\gamma^{-1} \alpha)^{q^{s(n+2)}}x^{q^{2s\frac{n+3}2}},$$
with $c \in \F_{q^n}$ and $\alpha^{q-1}=-1$.

Clearly, $f\circ x^{q^{-s(n-1)}}= c \gamma^{-1} \alpha x+(c\gamma^{-1} \alpha)^{q^{s(n+2)}}x^{q^{2s}}$.
It has a nonzero root if and only if
\[ \N_{q^{2n}/q^2}\left((c \alpha \gamma^{-1})^{1-q^{s(n+2)}} \right)=-1. \]
Since
\[\N_{q^{2n}/q^2}\left(c^{1-q^{s(n+2)}}\alpha^{1-q^{s(n+2)}} \gamma^{q^{s(n+2)}-1} \right) = -\N_{q^{2n}/q^2}\left( \gamma^{q-1} \right). \]
Therefore, $c \gamma^{-1} \alpha x+(c\gamma^{-1} \alpha)^{q^{s(n+2)}}x^{q^{2s}}=0$ has a no-zero solution, if and only if
\[ \N_{q^{2n}/q^2}\left( \gamma^{q-1} \right)=1, \]
which implies that $\N_{q^{2n}/q^2}\left( \gamma \right) \in \F_q$. This is a contradiction since $\N_{q^{2n}/q}(\gamma)$ is a non-square in $\F_q$.
\end{proof}

Finally, we prove that construction exhibited in Theorem \ref{thm:newhermitiancode}, is equivalent to none of the known examples with involved parameters.
We need the following tools from \cite{LTZ1}, used by the authors in order to solve the equivalence issue for the family of generalized twisted Gabidulin codes.

Let $\mathcal{C}$ be a subset of $\mathcal{L}_{n,q}$.
The \emph{universal support} $\mathcal{S}(\C)$ of $\C$ is the subset of $\{0,1,\ldots,n-1\}$ defined as follows
\[ \mathcal{S}(\C)=\{i\colon \text{there exists}\,\,f \in \C\,\,\text{such that the}\,\,q^i\text{-coefficient of}\,\, f \,\,\text{is not zero}\}, \]
whereas an \emph{independent support} $B$ is a subset of $\{0,1,\ldots,n-1\}$ for which there exists a set $\{h_i \colon i \in B\}$ of permutations of $\F_{q^{n}}$ such that
\[ \left\{ \sum_{i \in B} h_i(a)x^{q^i} \colon a \in \F_{q^n} \right\}\subseteq \C. \]
Also, let $A$ and $B$ two subsets of $\{0,1,\ldots,n-1\}$, then
\[ A^B:=\{k \colon \text{there exists a unique pair}\,\,(i,j)\in A\times B\,\,\text{such that}\,\, k\equiv i+j \pmod{n}\}. \]
For two extended equivalent codes the following holds.

\begin{lemma}\cite[Lemma 4.6]{LTZ1}\label{support}
Let $\C_1$ and $\C_2$ two subsets of $\mathcal{L}_{n,q}$.
Assume that $\C_1$ and $\C_2$ are extended equivalent, i.e.  $\tau(\C_1)=\C_2$ for some $\tau$ as in \eqref{equivrm}.
Let $A$ be the support of $\{\tau(ax)\colon a \in \F_{q^n}\}$. Then
\[ A^B\subseteq \mathcal{S}(\C_2), \]
for every independent support.
\end{lemma}

Now, we are ready to prove our final result.

\begin{theorem}
The $2$-code $\tilde{\mathcal{H}}_s$ is new.
\end{theorem}
\begin{proof}
We first remind that, by Theorem \ref{th:nodesiMir}, the $2$-code $\mathrm M$ described in (\ref{example:Myriam}), is not a $t$-design for any $t\neq 0$. Then, by Corollary \ref{cor:designproperty}, it is plain that $\tilde{\mathcal{H}}_s$ cannot be equivalent to $\mathrm M$.

On the other hand, assume by way of contradiction that $\tilde{\mathcal{H}}_s$ is extended equivalent to $\cH_{n,2,\ell}$. Since both codes are $(n-1)$-designs, as a direct consequence of Theorem \ref{th:equiv} and Corollary \ref{cor:equiv}, then they have to be equivalent in $\mathcal{H}_n(q^2)$, i.e. there must be a map of type $\Theta_{a,g^{\top q},\rho}$ such that $\Theta_{a,g^{\top q},\rho}(\tilde{\cH}_s)=\cH_{n,2,\ell}$, for given $a \in \F^*_{q}$, $\rho \in \Aut(\F_{q^2})$, and $g(x)=\sum_{i=0}^{n-1} g_i x^{q^{2i}}$ a permutation $q^2$-polynomial over $\F_{q^{2n}}$.

In what follows we will first prove that under this assumption, it must necessarily be $\ell \equiv \pm s \,(mod\,n)$. In fact, suppose that $\ell \not\equiv \pm s$. As $n$ is odd, we mast have that $(\ell,n)=1$, and hence there must be an $1<l<n-1$ such that $s\equiv l\ell \,(mod\,n)$.
\medskip

\noindent Let $A$ be the universal support of $\{g^{\top q}\circ a x \circ g(x) \colon a \in \F_{q^{2n}}\}$, and ${\mathcal S}(\cH_{n,2,\ell})$ be the universal support of $\cH_{n,2,\ell}$. By applying Lemma \ref{support} we get that $A^B \subseteq \mathcal{S}(\cH_{n,2,\ell})$ for each set of independent supports $B$ of $\tilde{\cH}_s$.

Now, consider the set \[\bigg \{is, (2n-2i+1)s\, \colon \, i=1,2,...,\frac{n-1}{2}\bigg \} \bigg \},\] which is a set of independent supports of $\tilde{\cH}_s$.

If $j\in A$, applying again Lemma \ref{support}, we get that \[\bigg \{ j+is:\, j+(2n-2i+1)s \,\colon \, i \in \bigg \{1,2,...,\frac{n-1}{2}\bigg \} \bigg \} \subseteq {\mathcal S}(\cH_{n,2,\ell}).\]

Hence, \[\bigg \{j+is; \,\,j+ (2n-2i+1)s \,\,\colon \, i \in \bigg \{1,2,...,\frac{n-1}{2} \bigg \} \bigg \} \subseteq \bigg \{i\ell; (2n-i+1)\ell \,\colon \, i \in \bigg \{1,2,...,\frac{n-1}{2} \bigg \} \bigg \}. \]

Letting \,$j \equiv u\ell \,(mod\,n)$ with $u\in \bigg \{1,2,...,\frac{n-1}{2}\bigg \}$ in above equation, and plugging in $s \equiv l\ell \,(mod\,n)$, we get $$\bigg \{u+il \,\text{ and } u+l(2n-i+1) \, \colon \, i=1,...,\frac{n-1}{2}\bigg \} \subseteq \bigg \{i, n-i+1 \,\colon \, i=1,...,\frac{n-1}{2}\bigg \}.$$ But since $l \geq 2$ and $u\in \bigg \{1,2,...,\frac{n-1}{2}\bigg \}$, this can never be the case. Hence, we end up with $\ell \equiv \pm s \,(mod\,n)$.

In this case consider the map $g^{\top q}\circ b^\rho x^{q^{2s\frac{n+1}{2}}} \circ g$. A direct computation shows that the coefficient of the term with $q$-degree $q^{2s\frac{n+1}{2}}$ in it, equals to

\begin{equation*}
a_{\frac{n+1}{2}}(b)=\sum^{n-1}_{i=0} g^{q^{s(2n-2i+1)}}_{i}\,\,g_{i}^{q^{s2(\frac{n+1}{2} -i)}}b^{\rho q^{s2(n-i)}}.
\end{equation*}

Since $(s,2n)=1$, the coefficients $g^{q^{s(2n-2i+1)}}_{i}\,\,g_{i}^{q^{s2(\frac{n+1}{2} -i)}}$ belongs to $\F_{q^n}$. As the coefficient of the term with $q^{2\ell}$-degree $\frac{n+1}{2}$ in $\cH_{n,2,\ell}$ is zero, and since $\ell \equiv \pm s \pmod{n}$, we get that $a_{\frac{n+1}{2}}(b)$ must be zero for each $b \in \F_{q^n}$. But this finally contradicts the fact that $g$ is a permutation polynomial.

Hence, we may conclude that $\tilde{\cH}_s$ is equivalent to none of the two existing examples with the involved parameters.
\end{proof}

\section{Concluding remarks and open problems}

In this article we provide some conditions ensuring the identification of the two types of equivalences which can be naturally defined for maximum additive $d$-codes in the Hermitian association scheme. More precisely in Theorem \ref{th:equiv} we prove that the equivalence and the extended equivalence coincide for maximum additive Hermitian $d$-codes with $d<n$ which are also $(n-d)$-designs. As a byproduct, in Corollary \ref{cor:equiv}  we prove that the equivalence and the extended equivalence coincide, whenever we deal with two maximum additive Hermitian $d$-codes with $d<n$ and $d$ odd.
However, it is an open question whether or not this holds true also for maximum additive Hermitian $d$-codes with $d<n$ and $d$ even, which are not $(n-d)$-designs.

Also, it would be interesting to understand whether the same result holds for maximum additive codes in the alternating setting.
In addition, we do not know whether Lemma \ref{tec} may be used for constructing new examples of $2$-codes in such a context.

Furthermore, one of the most important open problems regards the construction of maximum Hermitian $d$-codes for $3<d<n-1$ with $n$ and $d$ both even.
Probably, further investigations on the relations between the coefficients of a linearized polynomial and the dimension of its kernel (i.e. by using results contained in \cite{qres,teoremone,PZ}) may lead to new constructions for some fixed value of $n$.

\noindent Rocco Trombetti\\
Dipartimento di Matematica e Applicazioni ``Renato Caccioppoli"\\
Università degli Studi di Napoli ``Federico II",\\
Via Cintia, Monte S.Angelo I-80126 Napoli, Italy\\
{{\em rtrombet@unina.it}}

\smallskip

\noindent Ferdinando Zullo\\
Dipartimento di Matematica e Fisica,\\
Universit\`a degli Studi della Campania ``Luigi Vanvitelli'',\\
Viale Lincoln, 5 \\
I--\,81100 Caserta, Italy\\
{{\em ferdinando.zullo@unicampania.it}}


\begin{thebibliography}{1000}

\bibitem{BZZ}
{\sc D. Bartoli, C. Zanella and F. Zullo: }
A new family of maximum scattered linear sets in $\PG(1,q^6)$, \href{https://arxiv.org/abs/1910.02278}{arXiv:1910.02278}.

\bibitem{BrouwerCohen}
{\sc A.E. Brouwer, A.M. Cohen and A. Neumaier:}
Distance-regular graphs,
\emph{Springer}, Berlin (1989).

\bibitem{CarlitzHodges}
{\sc L. Carlitz and J.H. Hodges:}
Representations by Hermitian forms in a finite field,
\emph{Duke Math. J.} {\bf 22} (1955), 393--405.

\bibitem{qres}
{\sc B. Csajb\'ok:} Scalar $q$-subresultants and Dickson matrices,
\emph{J. Algebra} {\bf 547} (2020), 116--128.

\bibitem{CMPZ}
{\sc B. Csajb\'ok, G. Marino, O. Polverino and C. Zanella:}
A new family of MRD-codes,
\emph{Linear Algebra Appl.} {\bf 548} (2018), 203--220.

\bibitem{CsMPZh}
{\sc B. Csajb\'ok, G. Marino, O. Polverino and Y. Zhou:}
Maximum Rank-Distance codes with maximum left and right idealisers, to appear in \emph{Discrete Mathematics},
\href{https://arxiv.org/abs/1807.08774}{https://arxiv.org/abs/1807.08774}.

\bibitem{CSMPZ2016}
{\sc B. Csajb\'ok, G. Marino, O. Polverino and F. Zullo:}
Maximum scattered linear sets and MRD-codes,
{\it J.\ Algebraic\ Combin.} {\bf 46} (2017), 1--15.

\bibitem{teoremone}
{\sc B. Csajb\'ok, G. Marino, O. Polverino and F. Zullo:}
A characterization of linearized polynomials with maximum kernel,
\emph{Finite Fields Appl.} {\textbf{56}} (2019), 109--130.

\bibitem{CsMZ2018}
{\sc B. Csajb\'ok, G. Marino and F. Zullo:}
New maximum scattered linear sets of the projective line,
\emph{Finite Fields Appl.} {\bf 54} (2018), 133--150.

\bibitem{Coulter_Henderson}
{\sc R.S. Coulter, M. Henderson:}
Commutative  presemifields  and  semifields,
{\it Advances in Mathematics} {\bf 217(1)}, (2008).

\bibitem{Delsarte}
{\sc P. Delsarte:}
Bilinear forms over a finite field, with applications to coding theory,
{\it J.\ Combin.\ Theory Ser.\ A} {\bf 25(3)} (1978), 226--241.

\bibitem{DelsarteGoethals}
{\sc P. Delsarte and J.M. Goethals:}
Alternating bilinear forms over $GF(q)$,
{\it J.\ Combin.\ Theory Ser.\ A} {\bf 19(1)} (1975), 26--50.

\bibitem{Gabidulin}
{\sc E. Gabidulin:}
Theory of codes with maximum rank distance,
\emph{Problems of information transmission} {\bf 21(3)} (1985), 3--16.

\bibitem{GQ2009}
{\sc R. Gow and R. Quinlan:}
Galois extensions and subspaces of alterning bilinear forms with special rank properties,
{\it Linear Algebra Appl.} {\bf 430} (2009), 2212--2224.

\bibitem{GQ2009x}
{\sc R. Gow and R. Quinlan:}
Galois theory and linear algebra,
{\it Linear Algebra Appl.} {\bf 430} (2009), 1778--1789.

\bibitem{LN2016}
{\sc D. Liebhold and G. Nebe}:
Automorphism groups of Gabidulin-like codes,
{\it Arch. Math.} {\bf107(4)} (2016), 355--366.

\bibitem{LLTZ}
{\sc G. Longobardi, G. Lunardon, R. Trombetti and Y. Zhou:}
Automorphism groups and new constructions of maximum additive rank metric codes with restrictions,
\emph{Discrete Math.} {\bf 343(7)} (2020).

\bibitem{LMPTSymp}
{\sc G. Lunardon, G. Marino, O. Polverino and R. Trombetti:}
Symplectic semifield spreads of $\mathrm{PG}(5,q)$ and the veronese surface, \emph{Ricerche di matematica} {\bf 60(1)} (2011), 125--142.

\bibitem{LTZ1}
{\sc G. Lunardon, R. Trombetti and Y. Zhou:}
Generalized twisted gabidulin codes,
\emph{J. Combin. Theory Ser. A} {\bf 159} (2018), 79--106.

\bibitem{LTZ2}
{\sc G. Lunardon, R. Trombetti and Y. Zhou:}
On kernels and nuclei of rank metric codes,
{\it J. Algebraic Combin.} {\bf 46} (2017), 313--340.

\bibitem{MMZ}
{\sc G. Marino, M. Montanucci and F. Zullo:}
MRD-codes arising from the trinomial $x^q + x^{q^3}+ cx^{q^5} \in \mathbb{F}_{q^6}[x]$,
\emph{Linear Algebra Appl.} \textbf{591} (2020), 99--114.

\bibitem{PZ}
{\sc O. Polverino and F. Zullo:}
On the number of roots of some linearized polynomials,
\emph{Linear Algebra Appl.} {\bf 601} (2020), 189--218.

\bibitem{Sheekey2016}
{\sc J. Sheekey:}
A new family of linear maximum rank distance codes,
{\it Adv. Math. Commun.} {\bf 10(3)} (2016), 475--488.

\bibitem{MiriamSchmidt_master'sthesis}
{\sc M. Schmidt:}
Rank Metric Codes,
\emph{Master's thesis in Mathematics}.

\bibitem{Schmidt2010}
{\sc K.-U. Schmidt:}
Symmetric bilinear forms over finite fields of even characteristic,
\emph{J. Combin. Theory Ser. A} {\bf 117(8)} (2010), 1011--1026.

\bibitem{Schmidt2015}
{\sc K.-U. Schmidt:}
Symmetric bilinear forms over finite fields with applications to coding theory,
\emph{J. Algebraic Combin.} {\bf 42(2)} (2015), 635--670.

\bibitem{Schmidt2018}
{\sc K.-U. Schmidt:}
Hermitian rank distance codes,
\emph{Des. Codes Cryptogr.} {\bf 86(7)} (2018), 1469--1481.

\bibitem{SchmidtZhou}
{\sc K.U. Schmidt and Y. Zhou:}
On the number of inequivalent MRD codes,
{\it Des. Codes Cryptogr.} {\bf 86(9)} (2018), 1973--1982.

\bibitem{Stanton}
{\sc D. Stanton:}
A partially ordered set and $q$-Krawtchouk polynomials,
\emph{J. Combin. Theory Ser. A} {\bf 30(3)} (1981), 276--284.

\bibitem{TZ}
{\sc R. Trombetti and Y. Zhou:}
A new family of MRD codes in $\F_q^{2n\times 2n}$ with right and middle nuclei $\F_{q^n}$,
{\em IEEE Trans. Inform. Theory} {\bf 65(2)} (2019), 1054--1062.

\bibitem{Wan}
{\sc Z. X. Wan:}
Geometry of Matrices,
\emph{World Scientific} (1996).

\bibitem{ZZ}
{\sc C. Zanella and F. Zullo:}
Vertex properties of maximum scattered linear sets of $\PG(1,q^n)$,
\emph{Discrete Math.} {\bf 343(5)} (2020).

\bibitem{Zhou}
{\sc Y. Zhou:}
On equivalence of maximum additive symetric rank-distance codes,
\emph{Des. Codes Cryptogr.} {\bf 88} (2020), 841--850.

\bibitem{Zhou_Pott}
{\sc Y. Zhou, A. Pott:}
A new family of semifields with 2 parameters,
{\it Advances in Mathematics} {\bf 234}, (2013), 43-60.


\end{thebibliography}
\end{document}